\documentclass{article}
\usepackage[a4paper]{geometry}

\usepackage{mathtools}
\usepackage{amssymb}
\usepackage{amsthm}
\usepackage{ dsfont }
\usepackage{tikz}
\usetikzlibrary{decorations.markings}
\usetikzlibrary{decorations.pathreplacing}
\usetikzlibrary{arrows}
\usetikzlibrary{shapes.misc}
\usepackage{hyperref}

\newenvironment{rules}
{\begin{list}{\(\cdot\)}{
\itemsep=0em
\leftmargin=5mm
\labelwidth=3mm
\labelsep=3mm
}}{\end{list}}

\newtheorem{lemma}{Lemma}[section]
\newtheorem{corollary}[lemma]{Corollary}
\newtheorem{prop}[lemma]{Proposition}
\newtheorem{thm}[lemma]{Theorem}
\newtheorem{thm*}{Theorem}

\theoremstyle{definition}

\newtheorem{Def}[lemma]{Definition}
\newtheorem*{Def*}{Definition}
\newtheorem*{notation}{Notation}
\newtheorem*{convention}{Convention}
\newtheorem{rmk}[lemma]{Remark}
\newtheorem{construction}[lemma]{Construction}

\theoremstyle{remark}

\newcommand{\cuco}[1]{{\mathcal #1}}

\def\RR{\mathbb{R}}

\def\H{\mathcal{H}}
\def\F{\mathcal{F}}

\def\1{\mathds{1}}
\def\I{\mathfrak{L}}
\def\S{\mathfrak{S}}

\def\MCG{\mathcal{Q}}

\def\T{\mathcal{T}}
\def\Pyr{\widehat{CS}}
\def\Z{\mathfrak{F}}

\def\GG{\mathcal{G}}
\def\X{\mathcal{X}}
\def\Y{\mathcal{Y}}

\def\Int{[0,\infty)}

\let\temp\epsilon
\let\epsilon\varepsilon
\let\varepsilon\temp

\title{Hyperbolic HHS II: Graphs of hierarchically hyperbolic groups}
\author{Davide Spriano}
\date{January 5, 2018}

\begin{document}

\maketitle

\abstract{In this paper we consider a large family of graphs of hierarchically hyperbolic groups (HHG) and show that their fundamental groups admit HHG structures.  To do that, we will investigate the notion of \emph{hierarchical quasi convexity} and show that for a hyperbolic HHS it coincides with the notion of quasi-convexity. The main technical result, for which we expect further applications, is that it is possible to incorporate the HHG structure of a hierarchically hyperbolically embedded subgroup into the HHG structure of the ambient group. 
This generalizes and provides some additional details to a procedure described in \cite{HHSBoundaries}.}

\tableofcontents

\section{Introduction}

The study of hierarchically hyperbolic spaces and groups (respectively HHS and HHG) was introduced by Behrstock, Hagen and Sisto in \cite{HHSI}.
HHS and HHG form a very large class of spaces which contains several examples of interest including mapping class groups, right-angled Artin groups, many CAT(0) cube complexes, most 3-manifold groups, Teichm\"uller space (in any of the standard metrics), etc. A lot is known about HHS, see \cite{HHSII,HHSAsdim2015,HHSFlats,HHSBoundaries}. For instance a remarkable result about hierarchically hyperbolic space is that, under very mild conditions, every top-dimensional quasi-flat in an HHS lies within  finite Hausdorff distance of a finite union of standard orthants. This proves open conjectures of Farb in the case of Mapping Class Group, and Brock in the case of Teichm\"uller space.
For an introductory survey on HHS and HHG, we refer to \cite{HHSSurvey}. 

Given a class of objects of interest, in this case the class of HHG, it is natural to ask if it is possible to obtain new examples by combining objects in the class.
In the case of groups, one of the most natural ways to combine them is to take the fundamental group of an "admissible" graph of groups. In the context of hyperbolic groups, the most comprehensive answer in this sense is provided by the celebrated Bestvina-Feighn Combination Theorem (\cite{BestvinaFeighnCombination, BestvinaFeighnCombinationAddendum}) that provides sufficient conditions
which ensure that the fundamental group of a finite graph of hyperbolic groups is word-hyperbolic.
In \cite{HHSII}, Behrstock, Hagen and Sisto introduced the definition of graph of HHG and proved what we will refer to as \emph{the Combination Theorem for HHG} (\cite[Corollary 8.22]{HHSII}, see Theorem \ref{thm: combination HHG}) which provides strong answers in this sense. Indeed, under some hypotheses on the edge maps, they showed that the fundamental group of a graph of HHG is indeed an HHG.

The idea of the Combination Theorem is that, given a graph of HHG \(\GG\), some Cayley graphs of the various edge and vertex groups can be arranged in a "Bass-Serre tree" fashion, giving a space \(\T\) on which \(\pi_1(\GG)\) acts.
The hypotheses basically guarantee that the HHS structures of the various groups glue together to get an HHS structure on \(\T\) that is \(\pi_1(\GG)\) equivariant. 

The main goal of this paper is to show that a large class of graphs of HHG satisfies the hypotheses of the Combination Theorem. One of the key tools that will be used is the concept of \emph{hierarchically hyperbolically embedded subgroups}. Hierarchically hyperbolically embedded subgroups (introduced in \cite{HHSAsdim2015}) are the natural generalization of hyperbolically embedded subgroups to the hierarchical setting. Hyperbolically embedded subgroups were introduced and extensively studied in \cite{DGO} and can be thought of as a generalization of the peripheral structure of a relatively hyperbolic group. The existence of a (non degenerate) hyperbolically embedded subgroup has far reaching consequence, such as acylindrical hyperbolicity of the ambient group (\cite{OsinAcyl}). Conversely, given an acylidrically hyperbolic group, there are constructions that provides several hyperbolically embedded subgroups \cite{DGO}.
Moreover, hyperbolically embedded subgroups are "generic" as showed by Maher and Sisto in \cite{SistoMaher}.
For these reasons, considering hierarchically hyperbolically embedded subgroups should not be thought of as an exceedingly restrictive hypothesis.

Intuitively, the main result of this paper consists in showing that the hypotheses of the Combination Theorem are stable under adding new vertices with hyperbolic vertex groups and whose edge groups are hierarchically hyperbolically embedded in the vertex groups. Before giving a more precise formulation of the main result, let's consider a motivating example.
It is easy to see that a graph of HHG \(\GG\) consisting of a single vertex \(v\) satisfies the hypotheses of the Combination Theorem (Remark \ref{rmk:singleton graph}). Let \(\MCG\) be the vertex HHG associated to \(v\).
Then we can transform the graph \(\GG\) adding a vertex \(w\) and an edge \(e\) between \(v\) and \(w\) such that \(G= G_w\) is hyperbolic and \(H= H_e\) is a quasi-convex subgroup of \(G\). We require moreover that \(H\) is hierarchically hyperbolically embedded in \(\MCG\). Since the starting graph satisfy the Combination Theorem, so does the modified graph of groups. Thus the amalgamated product \(\MCG \ast_H G\) admits an HHG structure. 

It is worth underlining that the only hierarchical-like requirement concerns the embedding \(H \hookrightarrow \MCG\). Since there are no requirements on the HHG structure of \(\MCG\), an hypothesis of this type is clearly needed. 

We consider now the construction above in full generality.
\begin{Def*}
Let \(\GG'= \bigsqcup \GG_i\) be a finite union of finite graphs of HHG. A graph of groups \(\GG\) is \emph{star-obtainable} from \(\GG'\) if \(\GG\) can be obtained from \(\GG'\) by a finite number of the following moves.
\begin{enumerate}
\item Joining two vertices of \(\GG'\) by an edge \(e\) such that \(G_e\) is hyperbolic.
\item Adding a new vertex \(v\) and joining it with a finite set of vertices of \(\GG'\) such that 
		\begin{rules}
		\item the vertex group \(G_v\) is hyperbolic;
		\item the edge groups are quasi-convex subgroups of \(G_v\); 
		\end{rules}
\end{enumerate}
We say that \(\GG\) is \emph{hyperbolic-obtainable} from \(\GG'\) if it is star-obtainable and the following holds.
For each vertex \(v\) of \(\GG'\) the set of edge groups adjacent to \(v\) that correspond to edges of \(\GG - \GG'\) is hierarchically hyperbolically embedded in \(G_v\).
\end{Def*}

The main result of the paper is the following.

\begin{thm*}[\ref{thm: main application combination}]\label{Thm 1 intro}
Let \(\GG\) be a graph of groups that is hyperbolic-obtainable from \(\GG' = \bigsqcup \GG_i\), where each \(\GG_i\) is a finite graph of hierarchically hyperbolic groups that satisfies the hypotheses of the Combination Theorem for HHG.
Then there are HHG structures on the vertices and edges  groups of \(\GG\) such that \(\GG\) satisfies the hypotheses of the Combination Theorem. In particular, \(\pi_1 (\GG)\) admits a hierarchically hyperbolic group structure.
\end{thm*}

There are two key ingredients in the proof of Theorem \ref{Thm 1 intro}: 
the flexibility in the HHG structures on hyperbolic groups provided by \cite[Corollary 5.23]{SprianoHyperbolicHHSI} and a construction that allows to incorporate the HHG structure of a hierarchically hyperbolically embedded subgroup \(H\) into an HHG \(\MCG\) (Proposition \ref{prop:properties_of_the_new_structure}). 
The latter is greatly inspired, and in fact is a generalization of \cite[Section 6]{HHSAsdim2015}. It also provides a more detailed proof of the Large Link Lemma condition of \cite[Proposition 6.14]{HHSAsdim2015}.

\subsection*{Outline}
Section \ref{sec:Background} recalls the definition of hierarchically hyperbolic spaces and the some results related to cone-offs and factor systems. 
Section \ref{sec: Hierarchical quasi-convexity} shows the equivalence of quasi-convexity and hierarchical quasi-convexity in hyperbolic spaces. It also recalls some fundamental theorems of the theory of HHS such as the existence of hierarchy paths and the distance formula. 
Section \ref{section:Graphs of HHG} contains the main result of the paper. In  Subsection \ref{subsec:definitions and notations} we recall the definitions and some results connected to HHG and (hierarchically) hyperbolically embedded subgroups. In Subsection \ref{subsec: changing str I} we state the main result and applications, and structure the ideas of the proof. The bulk of the proofs and technicalities is moved to Subsection \ref{subsec: changing str II}.

\subsection*{Acknowledgments}
The author would like to thank Alessandro Sisto for suggesting the topic and for very helpful comments and suggestions.

\section{Background}\label{sec:Background}

\subsection{Cone-offs}\label{subsec: cone-offs}

We recall the following results.
\begin{Def}[Coning-off]
	Let \(\Gamma\) be a graph, \(H\) a connected subgraph of \(\Gamma\). 
	We define the \emph{cone-off} of \(\Gamma\) with respect to \(H\), and denote it by \(\widehat{\Gamma}\),
	as the graph obtained from \(\Gamma\) adding 
	an edge connecting each pair of vertices in \((H \times H) -\Delta_{H\times H}\), where \(\Delta_{H\times H}\) denotes the diagonal.
	We call the edges added in such a way \emph{\(H\)-components}.
	Similarly, the cone-off with respect to a family of connected subgraphs \(\H=\{H_i\}\) is obtained adding the 
	\(H_i\)-components for each \(H_i \in \H\). 
	An edge is an \(\H\)-component if it is a \(H_i\)-component for some \(H_i \in \H\).
\end{Def}

\begin{prop}[Kapovich-Rafi, Bowditch]\label{Kapovich-Rafi}
	Let \(\Gamma\) be a connected graph with simplicial metric \(d_\Gamma\) such that \((X,d_X)\) is \(\delta\)-hyperbolic. 
	Let \(K> 0\) and \(\H\) be a family of \(K\)-quasi-convex subgraphs of \(\Gamma\).
	Let \(\widehat{\Gamma}\) be the cone-off of \(\Gamma\) with respect to the family \(\H\).
	Then \(\widehat{\Gamma}\) is \(\delta'\)-hyperbolic (with respect to the path metric) 
	for some constant \(\delta' >0\) depending only on \(K\) and \(\delta\).  Moreover there exists \(H=H(K,\delta )>0\) such that whenever
	\(x,y\in V(\Gamma)\), \([x,y]_\Gamma\) is a \(d_\Gamma\)-geodesic from \(x\) to \(y\) in \(\Gamma\) and
	\( [x,y]_{\widehat{\Gamma}}\) is a \(d_{\widehat{\Gamma}}\)-geodesic from \(x\) to \(y\) in \(\widehat{\Gamma}\) 
	then \([x,y]_\Gamma\) and \([x,y]_{\widehat{\Gamma}}\) are \(H\)-Hausdorff close in \((\widehat{\Gamma},d_{\widehat{\Gamma}})\).
\end{prop}

\begin{Def}
Let \(\widehat{\Gamma}\) be the cone-off of a graph \(\Gamma\) with respect to a family of quasi-isometrically embedded  subgraphs \(\H\). Let 
	\(\gamma = u_1 *e_1 * \cdots * e_n * u_{n+1}\) be a path of \(\widehat{\Gamma}\), where each \(e_i\) is an \(H_i\)-component for some
	\(H_i \in \H\), and the \(u_i\) are (possibly trivial) segments of \(\Gamma\). 
	The \emph{embedded-de-electrification} \(\widetilde{\gamma}^e\)
	of \(\gamma\) is the concatenation \(u_1 * \eta_1 * \cdots * \eta_n *u_{n+1}\) where each \(\eta_i\) is a geodesic segment of \(H_i\) connecting 
	the endpoints of \(e_i\). If \(e_i\) was an \(H\)-component, we say that \(\eta_i\) is an \emph{\(H\)-piece}.
	A piece of \(\widetilde{\gamma}\) is an \emph{\(\H\)-piece} if it is an \(H\)-piece for some \(H \in \H\). 
\end{Def}

\begin{prop}[\cite{SprianoHyperbolicHHSI} Corollary 2.30]\label{prop: Bound Hausdorff distance}
Let \(\Gamma\) be a \(\delta\)-hyperbolic graph, \(\H\) a family of uniformly quasi-isometrically embedded subgraphs and \(\widehat{\Gamma}\) the cone-off of \(\Gamma\) with respect to \(\H\). 
Then there exist \(\tau_1 = \tau_1 (\delta, K)\) and \(\tau_2= \tau_2 (\delta, K)\) such that for each pair of points \(x,y \in \Gamma\) there exists a \(\tau_1\)-quasi-geodesic \(\gamma'\) of \(\widehat{\Gamma}\) with the property that for each  embedded-de-electrification \(\widetilde{\gamma}'\) of \(\gamma'\), \(\widetilde{\gamma}'\) is a \(\tau_2\)-quasi-geodesic of \(\Gamma\).
\end{prop}

\subsection{Hierarchically hyperbolic spaces}\label{subsec: HHS}

The goal of this section is to recall the definition  of hierarcahically hyperbolic spaces and some important properties. For additional details, the main reference is \cite{HHSII}. 
\begin{Def}[Hierarchically hyperbolic space, \cite{HHSII}]\label{defn:space_with_distance_formula}
The $q$-quasigeodesic space  $(\cuco X,d_{\cuco X})$ is a \emph{hierarchically hyperbolic space} if there exists $\delta\geq0$, an index set $\mathfrak S$, and a set $\{CW:W\in\mathfrak S\}$ of $\delta$--hyperbolic spaces $(C U,d_U)$,  such that the following conditions are satisfied:\begin{enumerate}
\item\textbf{(Projections.)}\label{item:dfs_curve_complexes} There is
a set $\{\pi_W\colon \cuco X\rightarrow2^{C W}\mid W\in\mathfrak S\}$
of \emph{projections} sending points in $\cuco X$ to sets of diameter
bounded by some $\xi\geq0$ in the various $C W\in\mathfrak S$.
Moreover, there exists $K$ so that each $\pi_W$ is $(K,K)$--coarsely
Lipschitz.

 \item \textbf{(Nesting.)} \label{item:dfs_nesting} $\mathfrak S$ is
 equipped with a partial order $\sqsubseteq$, and either $\mathfrak
 S=\emptyset$ or $\mathfrak S$ contains a unique $\sqsubseteq$--maximal
 element; when $V\sqsubseteq W$, we say $V$ is \emph{nested} in $W$.  We
 require that $W\sqsubseteq W$ for all $W\in\mathfrak S$.  For each
 $W\in\mathfrak S$, we denote by $\mathfrak S_W$ the set of
 $V\in\mathfrak S$ such that $V\sqsubseteq W$.  Moreover, for all $V,W\in\mathfrak S$
 with $V$ properly nested in $W$ there is a specified subset
 $\rho^V_W\subset C W$ with $\mathrm{diam}_{C W}(\rho^V_W)\leq\xi$.
 There is also a \emph{projection} $\rho^W_V\colon C
 W\rightarrow 2^{C V}$.  (The similarity in notation is
 justified by viewing $\rho^V_W$ as a coarsely constant map $C
 V\rightarrow 2^{C W}$.)
 
 \item \textbf{(Orthogonality.)} 
 \label{item:dfs_orthogonal} $\mathfrak S$ has a symmetric and
 anti-reflexive relation called \emph{orthogonality}: we write $V\bot
 W$ when $V,W$ are orthogonal.  Also, whenever $V\sqsubseteq W$ and $W\bot
 U$, we require that $V\bot U$.  Finally, we require that for each
 $T\in\mathfrak S$ and each $U\in\mathfrak S_T$ for which
 $\{V\in\mathfrak S_T\mid V\bot U\}\neq\emptyset$, there exists $W\in
 \mathfrak S_T-\{T\}$, so that whenever $V\bot U$ and $V\sqsubseteq T$, we
 have $V\sqsubseteq W$.  Finally, if $V\bot W$, then $V,W$ are not
 $\sqsubseteq$--comparable.
 
 \item \textbf{(Transversality and consistency.)}
 \label{item:dfs_transversal} If $V,W\in\mathfrak S$ are not
 orthogonal and neither is nested in the other, then we say $V,W$ are
 \emph{transverse}, denoted $V\pitchfork W$.  There exists
 $\kappa_0\geq 0$ such that if $V\pitchfork W$, then there are
  sets $\rho^V_W\subseteq C W$ and
 $\rho^W_V\subseteq C V$ each of diameter at most $\xi$ and 
 satisfying: $$\min\left\{d_{
 W}(\pi_W(x),\rho^V_W),d_{
 V}(\pi_V(x),\rho^W_V)\right\}\leq\kappa_0$$ for all $x\in\cuco X$.
 
 For $V,W\in\mathfrak S$ satisfying $V\sqsubseteq W$ and for all
 $x\in\cuco X$, we have: $$\min\left\{d_{
 W}(\pi_W(x),\rho^V_W),\mathrm{diam}_{C
 V}(\pi_V(x)\cup\rho^W_V(\pi_W(x)))\right\}\leq\kappa_0.$$ 
 
 The preceding two inequalities are the \emph{consistency inequalities} for points in $\cuco X$.
 
 Finally, if $U\sqsubseteq V$, then $d_W(\rho^U_W,\rho^V_W)\leq\kappa_0$ whenever $W\in\mathfrak S$ satisfies either $V\sqsubseteq W$ and \(V \neq W\) or $V\pitchfork W$ and $W\not\bot U$.
 
 \item \textbf{(Finite complexity.)} \label{item:dfs_complexity} There exists $n\geq0$, the \emph{complexity} of $\cuco X$ (with respect to $\mathfrak S$), so that any set of pairwise--$\sqsubseteq$--comparable elements has cardinality at most $n$.
  
 \item \textbf{(Large links.)} \label{item:dfs_large_link_lemma} There
exist $\lambda\geq1$ and $E\geq\max\{\xi,\kappa_0\}$ such that the following holds.
Let $W\in\mathfrak S$ and let $x,x'\in\cuco X$.  Let
$N=\lambda d_{_W}(\pi_W(x),\pi_W(x'))+\lambda$.  Then there exists $\{T_i\}_{i=1,\dots,\lfloor
N\rfloor}\subseteq\mathfrak S_W-\{W\}$ such that for all $T\in\mathfrak
S_W-\{W\}$, either $T\in\mathfrak S_{T_i}$ for some $i$, or $d_{
T}(\pi_T(x),\pi_T(x'))<E$.  Also, $d_{
W}(\pi_W(x),\rho^{T_i}_W)\leq N$ for each $i$. 
 
 \item \textbf{(Bounded geodesic image.)} \label{item:dfs:bounded_geodesic_image} For all $W\in\mathfrak S$, all $V\in\mathfrak S_W-\{W\}$, and all geodesics $\gamma$ of $C W$, either $\mathrm{diam}_{C V}(\rho^W_V(\gamma))\leq E$ or $\gamma\cap N_E(\rho^V_W)\neq\emptyset$. 
 
 \item \textbf{(Partial Realization.)} \label{item:dfs_partial_realization} There exists a constant $\alpha$ with the following property. Let $\{V_j\}$ be a family of pairwise orthogonal elements of $\mathfrak S$, and let $p_j\in \pi_{V_j}(\cuco X)\subseteq C V_j$. Then there exists $x\in \cuco X$ so that:
 \begin{itemize}
 \item $d_{V_j}(x,p_j)\leq \alpha$ for all $j$,
 \item for each $j$ and 
 each $V\in\mathfrak S$ with $V_j\sqsubseteq V$, we have 
 $d_{V}(x,\rho^{V_j}_V)\leq\alpha$, and
 \item if $W\pitchfork V_j$ for some $j$, then $d_W(x,\rho^{V_j}_W)\leq\alpha$.
 \end{itemize}

\item\textbf{(Uniqueness.)} For each $\kappa\geq 0$, there exists
$\theta_u=\theta_u(\kappa)$ such that if $x,y\in\cuco X$ and
$d(x,y)\geq\theta_u$, then there exists $V\in\mathfrak S$ such
that $d_V(x,y)\geq \kappa$.\label{item:dfs_uniqueness}
\end{enumerate}
We often refer to $\mathfrak S$, together with the nesting
and orthogonality relations, the projections, and the hierarchy paths,
as a \emph{hierarchically hyperbolic structure} for the space $\cuco
X$.  
\end{Def}

\begin{Def}[Factor System]\label{Factor System}
	Let \(\Gamma\) be a connected graph. We say that a family \(\H\) of connected subgraphs of \(\Gamma\) is a 
	\emph{factor system} for \(\Gamma\) if there are constants \(K,c, \xi, B\) such that the following are satisfied. 
	\begin{enumerate}
	\item Each \(H \in \H\) is \(K\)-quasi-isometrically embedded in \(\Gamma\).
	\item Given \(H_1, H_2 \in \H\), either \(\mathrm{diam} (p_{H_1} (H_2)) \leq \xi\), 
	or there is \(U \in \H\) with such that \(U \subseteq H_1\) and \(d_{\mathrm{Haus}}(p_{H_1} (H_2), U ) \leq B\).
	\item Given \(H_1, H_2 \in \H\), if \(d_{\mathrm{Haus}}(p_{H_1}(H_2), H_1) \leq B\), then \(H_1 \subseteq H_2 \).
	\item Every ascending chain of inclusions \(H_1 \subsetneq H_2 \subsetneq \cdots \subsetneq H_n\) has length bounded above by \(c\).
	\item Given \(H_1, H_2 \in \H\), if \(d_{\mathrm{Haus}}(H_1, H_2) < \infty\), then \(H_1 = H_2\).
	\end{enumerate}
\end{Def}

\begin{thm}[\cite{SprianoHyperbolicHHSI}, Theorem 3.14]\label{thm: HHS structure from fact syst}
Let \(\Gamma\) be a hyperbolic graph, \(\H\) be a factor system for \(\Gamma\) and, for each \(H \in \H \cup \{\Gamma\}\), let \(\widehat{H}\) be the cone off of \(H\) with respect to all elements of \(\H\) that are strictly contained in \(H\). 
Then \(\Gamma\) admits a hierarchically hyperbolic structure, with indexing set \(\H \cup \{\Gamma\}\), and with associated hyperbolic spaces the family \(\{\widehat{H} \mid H \in \H \cup \{\Gamma\}\}\). 
The projections are realized as closest point projections in \(\Gamma\).
\end{thm}

There is a particular instance of factor system that we will use in Section \ref{section:Graphs of HHG}.

\begin{corollary}\label{cor: Simple Factor System}
Let \(\Gamma\) be a hyperbolic graph and let \(\H\) be a family of uniformly quasi-isometrically embedded subgraphs of \(\Gamma\). Suppose that all the elements of \(\H\) have infinite diameter. Suppose, moreover,  that for each \(\epsilon\) exists \(R = R(\epsilon)\) such that for any pair of distinct elements \(H_1, H_2\) of \(\H\), the intersection \(N_\epsilon (H_1) \cap H_2\) has diameter at most \(R\).
Then \(\H\) is a factor system for \(\Gamma\) and \((\Gamma, \H \cup \{\Gamma\})\) is a hierarchically hyperbolic space. 
\begin{proof}
The axioms 1,4 and 5 are immediate from the hypotheses. Axiom 3 follows easily from hyperbolicity of \(\Gamma\). Consider Axiom 2. We claim that there is \(\xi\) such that for each pair of elements \(H_1\) and \(H_2\) of \(\H\), we have \(\mathrm{diam}(p_{H_2}(H_1)) \leq \xi\). 
Suppose this is not the case. Since \(\Gamma\) is hyperbolic, the spaces \(H_i\) are uniformly quasi-convex. In particular, there is are constants \(K\) and \(T\) such that \(\mathrm{diam}(p_{H_2} (H_1)) - T \leq \mathrm{diam}(N_K(H_1) \cap H_2) \). But the latter is uniformly bounded by \(R(K)\).  
\end{proof}
\end{corollary}

A consequence of a generalized version of Theorem \ref{thm: HHS structure from fact syst} is the following.

\begin{thm}[\cite{SprianoHyperbolicHHSI}, Corollary 5.23]\label{thm: HHG structure on G}
Let \(G\) be a hyperbolic group and let \(\F= \{F_1, \dots, F_N\}\) be a finite family of infinite quasi-convex subgroups. 
Let \(\sim\) be the equivalence relation between subset of \(G\) given by having finite distance in \(\mathrm{Cay}(G)\)  (note that does not depend on the choice of generators).
Then there exists a finite family of quasi-convex subgroups \(\F^{(M)}\) that contains \(\F\) such that if \(\F_{\mathrm{cos}}\) is the set of cosets of \(\F^{(M)}\), then \((G, \F_{\mathrm{cos}} /_\sim )\) is a hierarchically hyperbolic group structure on \(G\). 
\end{thm}

\begin{corollary}[\cite{SprianoHyperbolicHHSI}, Section 5]\label{cor: sub-fac-sys Gps}
Let \(G\) be a hyperbolic group, and let \(\F= \{F_1, \dots , F_N\}\) be a finite family of quasi-convex subgroups. 
Let \((G, \S)\) be the HHG structure on \(G\) provided by Theorem \ref{thm: HHG structure on G}. Then for each \(F \in \F\), we have that \((F, \S_{[F]})\) is an HHG structure on \(F\). Moreover the inclusion map \(i \colon F \hookrightarrow G\) induces an injective hieromorphism between \(F\) and \(G\), such that  for each \([H] \in \S_{[F]}\), we have \(i^\bigcirc ([H]) = [H]\).

\end{corollary}

\section{Hierarchical quasi-convexity}\label{sec: Hierarchical quasi-convexity}
We conclude this section providing a result that will be useful in Section \ref{section:Graphs of HHG}, namely that for a hyperbolic graph \(\Gamma\) equipped with the HHS structure coming from a factor system, quasi-convex subgraphs of \(\Gamma\) are  \emph{hierarchically quasi-convex}. The notion of hierarchical quasi-convexity was introduced by Behrstock, Hagen and Sisto and has deep consequences (see Sections 5,6,7 of \cite{HHSII}).

 \begin{Def}[Hierarchical quasi-convexity, \cite{HHSII} Definition 5.1]\label{def:hierarchically_q.c.}
Let \((\X, \S)\) be a hierarchically hyperbolic space. We say that a subset \(Y \subset \MCG\) is \emph{\(k\)-hierarchically quasi-convex}, for some \(k \colon [0,\infty) \rightarrow [0, \infty)\), if the following holds:
\begin{enumerate}
\item For each \(U \in \S\), one has that \(\pi_U(Y)\) is \(k(0)\)-quasi-convex in \(CU\).
\item For every \(r \geq 0\) and every \(x \in \X\) such that \(d_{CU}(\pi_U(x), \pi_U(Y)) \leq r\) for all \(U \in \S\), one has that \(d_\X (x, Y) \leq k(r)\).
\end{enumerate} 
\end{Def}

We recall the Distance Formula for hierarchically hyperbolic spaces

\begin{thm}[Distance formula, \cite{HHSII} Theorem 4.5]\label{thm:Distance_formula}
Let \((\X, \S)\) be hierarchically hyperbolic. Then there exists \(s_0\) such that for all \(s \leq s_0\)  there exists constants \(K, C\) such that for all \(x,y \in \X\),
\[d_{\X} (x,y) \asymp_{(K,C)} \sum_{W \in \S} \left\{\left\{ d_{CW}(\pi_W(x), \pi_W(y))\right\} \right\}_s,\]
where \(\{\{A\}\}_s\) is equal to \(0\) if \(A<s\), and equal to \(A\) otherwise. 
\end{thm}
 
\begin{Def}[Hierarchy path, \cite{HHSII} Definition 4.2]
Let \((\X, \S)\) be a hierarchically hyperbolic space. A path \(\gamma\) of \(\X\) is a \emph{D-hierarchy path} if 
\begin{enumerate}
\item \(\gamma\) is a \((D,D)\)-quasi-geodesic of \(\X\);
\item for each \(U \in \S\) the projection \(\pi_U(\gamma)\) is a \((D, D)\)-quasi-geodesic.
\end{enumerate}
\end{Def}

\begin{thm}[Existence of hierarchy paths, \cite{HHSII} Theorem 4.4]\label{thm: Existence of hierarchy paths}
Let \(\X\) be a hierarchically hyperbolic space. 
Then there exists \(D_0\) so that any \(x,y \in \X\) are joined by a \(D_0\)-hierarchy path.
\end{thm}

We can now state the main result of this section.

\begin{prop}[Equivalence of quasi-convexity and hierarchical quasi-convexity]\label{prop: Y is hierarchically quasi-convex}
Let \((\X, \S)\) be a hyperbolic hierarchically hyperbolic space, and let \(\Y\) be a subspace of \(\X\). Then there are functions \(Q\) and \(K\) such that the following holds. If \(\Y\) is \(q\)-quasi-convex in \(\X\), then \(\Y\) is \(K(q)\)-hierarchically quasi convex in \((\X, \S)\). Conversely, if \(\Y\) is \(k\)-hierarchically quasi-convex in \((\X, \S)\), then \(\Y\) is \(Q(k)\)-quasi-convex in \(\X\).
\begin{proof}
The fact that hierarchical quasi-convexity implies quasi-convexity is an easy consequence of Theorems \ref{thm: Existence of hierarchy paths} and \ref{thm:Distance_formula}.
Suppose that \(\Y\) is hierarchically \(k\)-quasi-convex for some \(k \colon \Int \rightarrow \Int\), let \(a, b\) be two points of \(\Y\), and let \([a,b]\) be a geodesic of \(\X\) connecting them. Let \(\gamma\) be a hierarchy path connecting \(a\) and \(b\), whose existence is guaranteed by Theorem \ref{thm: Existence of hierarchy paths}. Then, using the fact that the spaces \(CU\) are hyperbolic, we have that for each \(s \in \gamma\) and \(U \in \S\), we can uniformly bound the distance \(d(\pi_U(s), \pi_U(\Y))\). By the distance formula (Theorem \ref{thm:Distance_formula}), we can uniformly bound the distance between \(s\) and \(\Y\). Since \(\gamma\) is a quasi-geodesic, it has uniformly bounded Hausdorff distance from \([a,b]\), thus the result follows.

For the other implication, recall that hierarchically quasi-convexity consists of two parts. 
\textit{For each \(U \in \S\), the spaces \(\pi_U(\Y)\) are uniformly quasi-convex.}
Let \(x', y'\) be two points in \(\pi_U(\Y)\) and let \(x,y \in \Y\) be such that \(x' \in\pi_U(x)\) and \(y' \in \pi_U(y)\). Let \(\gamma\) be a hierarchy path between \(x\) and \(y\), whose existence is guaranteed by Theorem \ref{thm: Existence of hierarchy paths}.
Since \(\X\) is hyperbolic and \(\Y\) is quasi-convex, each point of \(\gamma\) has uniformly bounded distance from \(\Y\). Since projections are uniformly quasi-Lipschitz, this is still true for \(\pi_U(\gamma)\) and \(\pi_U(\Y)\). Since every \(CU\) is uniformly hyperbolic, the claim follows.

\textit{There is a function \(k\) such that if \(d_{CU}(x,\Y)\leq r\) for each \(U \in \S\), then \(d_{\X}(x,y)\leq k(r)\).}
Let \(x \in \X\) be such that for each \(U \in \S\), one has \(d_{CU}(x, \Y) \leq r\), and let \(y \in \Y\) such that \(d_{\X}(x,y)\) is minimal. 
If there is \(\psi (r)\) such that for all \(U \in \S\) one has \(d_{CU}(x,y)\leq \psi(r)\), then the result will follow form the distance formula (Theorem \ref{thm:Distance_formula}).
Choose \(U\) and let \(z\) be a point of \(\Y\) such that \(d_{CU}(x,z) = d_{CU}(x,\Y)\). We want to uniformly bound \(d_{CU}(y,z)\), which will imply the claim. 

Let \([x,y]\) be a geodesic of \(\X\) between \(x\) and \(y\). Since \([x,y]\) is a projection geodesic and \(\Y\) is quasi-convex in \(\X\) there is a point \(s' \in [x,z]\) that has uniformly bounded distance from \(y\).
Let \( \gamma\) be a hierarchy path of \(\X\) between \(x\) and \(y\).  Since \(\X\) is hyperbolic, the Hausdorff distance between \([x,z]\) and the hierarchy path \(\gamma\) is uniformly bounded. Thus there is a point \(s \in \gamma\) that has uniformly bounded distance from \(y \) in \(\X\). 
We remark that the uniform bound does not depend on \(x\) or \(U \in \S\). 

Since the projection \(\pi_U\) is quasi-Lipschitz, we can uniformly bound \(d_{CU} (s, y)\)by some uniform \(c\). Since \(\gamma\) is a hierarchy path, \(s\) lies at uniformly bounded distance \(\Lambda\) from a geodesic \([x,z]\) of \(CU\). Since \([x,z]\) is a projection geodesic, we get \(d_{CU}(y,z) \leq 2 (\Lambda + c)\). 
\end{proof}
\end{prop}

\section{Graphs of Hierarchically Hyperbolic Groups}\label{section:Graphs of HHG}
The goal of this section is to apply the constructions of the previous sections in order to get an HHG structure on certain fundamental groups of graphs of groups. Indeed, the Combination Theorem for HHG (see \cite{HHSII}) yields sufficient conditions for the fundamental group of a graph of groups to be an HHG.
We will check that these conditions are verified for a class of graphs of groups. 
The key step will be to show that given a finite family of hyperbolic HHGs \(\{(H_i, \H_i)\}_{i\in I}\) that is hierarchically hyperbolic embedded into an HHG \((\MCG, \S)\), then it is possible to modify the HHG structure on \(\MCG\) in such a way that the various inclusions \(\phi_i: H_i \hookrightarrow \MCG\) become homomorphisms of HHG.F

We will need some definitions introduced in \cite{HHSII}. We refer to \cite{HHSII} for additional details. 

\subsection{Definitions and notations}\label{subsec:definitions and notations}

\subsubsection{Morphism of hierarchically hyperbolic spaces}

\begin{Def}[Hieromorphism, \cite{HHSII} Definition 1.19]
Let $(\X,\mathfrak S)$ and $(\X',\mathfrak S')$ be
hierarchically hyperbolic structures on the spaces $\X,\X'$
respectively.  A \emph{hieromorphism}, 
consists of a map $f\colon \X\rightarrow\X'$, an injective map $f^\bigcirc \colon\mathfrak S\rightarrow\mathfrak
S'$ preserving nesting, transversality, and orthogonality, and maps
$f^* (U)\colon C U\rightarrow C(f^\bigcirc (U))$, for each \(U \in \S\),
which are uniformly quasi-isometric embeddings.
The three maps should preserve the structure of the hierarchically hyperbolic space, that is, they coarsely commute with the maps \(\pi_U\) and \(\rho^U_V\), for \(U, V\) in either \(\S\) or \(S'\), associated to the hierarchical structures.
\end{Def}

\begin{Def}[Automorphism, hierarchically hyperbolic group, \cite{HHSII} Definition 1.20]\label{defn:hhs_automorphism}
An \emph{automorphism} of the hierarchically hyperbolic space 
$(X,\mathfrak S)$ is a hieromorphism 
$f\colon (\X,\mathfrak S)\rightarrow(\X,\mathfrak S)$ such   
that $f^\bigcirc$ is bijective and each $f^*(U)$ is an 
isometry.

The finitely generated group $\MCG$ is \emph{hierarchically hyperbolic}
if there exists a hierarchically hyperbolic space $(\X,\mathfrak
S)$ on which \(\MCG\) acts by automorphisms of hierarchically hyperbolic spaces,  so that the uniform quasi-action of $G$ on $\X$ is metrically proper and cobounded
and $\mathfrak S$ contains finitely many $\MCG$--orbits.  Note that if
$\MCG$ is hierarchically hyperbolic by virtue of its action on the
hierarchically hyperbolic space $(\X,\mathfrak S)$, then
$(\MCG,\mathfrak S)$ is a hierarchically hyperbolic structure with
respect to any word-metric on $\MCG$; for any $U\in\mathfrak S$ the 
projection is the
composition of the projection $\X\rightarrow C
U$ with a $\MCG$--equivariant quasi-isometry
$\MCG\rightarrow\X$.  In this case, $(\MCG,\mathfrak S)$ (with the
implicit hyperbolic spaces and projections) is a \emph{hierarchically
hyperbolic group structure}.
\end{Def}

\begin{Def}[Full hieromorphim, \cite{HHSII} ]\label{def:full}
We say that a hieromorphism is \emph{full} if
\begin{rules}
 \item there exists $\xi\geq 0$ such that each $f^*(U)\colon C
 U\to C(f^\bigcirc (U))$ is a $(\xi,\xi)$--quasi-isometry, and
 \item for each $U\in\mathfrak S$, if $V'\in\mathfrak S'$ satisfies
 $V'\sqsubseteq f^\bigcirc (U)$, then there exists $V\in\mathfrak S$ such that
 $V\sqsubseteq U$ and $f^\bigcirc (V)=V'$.
\end{rules}  

As the functions $f, 
f^*(U),$ and $f^\bigcirc S$ all have distinct domains, it is 
often clear from the 
context which is the relevant map; in that case we periodically abuse 
notation slightly by dropping the superscripts and just calling all of the maps $f$.

\end{Def}

\begin{rmk}
Observe that the second item of the definition of full hieromorphism holds automatically unless $V'$ is bounded.
\end{rmk}

\begin{Def}[Homomorphism of hierarchically hyperbolic groups, \cite{HHSII} Def 1.21]\label{defn:homomorphism_of_HHGs}
Let $(\MCG,\mathfrak S)$ and $(\MCG',\mathfrak S')$ be
hierarchically hyperbolic groups. A \emph{homomorphism of hierarchically hyperbolic groups} consists of a hieromorphims \(\phi \colon \MCG \rightarrow \MCG'\) such that \(\phi\) is a homomorphism of group and the hieromorphism maps \(\phi^*, \phi^\bigcirc\) are uniformly coarsely-equivariant with respect to the homomorphism \(\phi\). 
\end{Def}

\subsubsection*{Hyperbolically embedded subgroups}

We recall the definition of hyperbolically embedded subgroups that was originally introduced by Dahmani, Guirardel and Osin in \cite{DGO}.

Let \(G\) be a group and \(\{H_i\}_{i\in I}\) be a family of subgroups of \(G\). Let \(T\) be a (possibly infinite) set of elements of \(G\) and let \(F=\bigsqcup_{i \in I} \left( H_i - \1\right)\). Suppose that \(T\) is symmetric (that is for each \(t \in T\), we have that \(t^{-1} \in T\)) and that \(\left< T \cup F \right> = G\). 
For each \(i \in I\), the Cayley graph \(\mathrm{Cay}(H_i, H_i - \{\1\})\) is a complete subgraph of \(\mathrm{Cay}(G, T \cup F)\). Let \(E(H_i)\) denote the set of edges of the Cayley graph of \(H_i\). 
\begin{Def}[\cite{DGO} Definition 4.2] 
For each \(i \in I\), we define the \emph{relative metric} \(\widehat{d}_i: H_i \times H_i \rightarrow [0, +\infty]\) as the length of the shortest path in \(\mathrm{Cay}(G, T \cup F) - E(H_i)\). If such a path does not exist, then we set the length to be \(+ \infty\).
\end{Def}

\begin{Def}[Hyperbolically embedded subgroups, \cite{DGO} Definition 4.25]
The family \(\{H_i\}_{i \in I}\) is \emph{hyperbolically embedded in \(G\)} with respect to the set \(T\), and denoted by \(\{H_i\} \hookrightarrow_h (G,T)\) if:
\begin{rules}
\item the graph \(\mathrm{Cay}(G, T \sqcup H)\) is hyperbolic;
\item for each \(i \in I\), the metric space \((H_i, \widehat{d}_i)\) is proper.
\end{rules}
We say that the family \(\{H_i\}\) is \emph{hyperbolically embedded in \(G\)} if there is a set \(X\) such that \(\{H_i\} \hookrightarrow_h (G, X)\). 
\end{Def}

\begin{Def}[Hierarchically hyperbolically embedded subgroups, \cite{HHSAsdim2015}]
Let \((\MCG, \S)\) be an HHG, let \(S\) be the \(\sqsubseteq\)-maximal element of \(\S\) and let \((\{H_i\}_{i\in I})\) be a finite family of subgroups of \(\MCG\). 
We say that the family \(\{H_i\}\) is \emph{hierarchically hyperbolically embedded} into \(\MCG\), and denote this by \(\H \hookrightarrow_{hh} \MCG\), if the following hold:
\begin{rules}
\item There is a (possibly infinite) generating set \(T\) of \(\MCG\) such that \(CS = \mathrm{Cay}(\MCG, T)\) and \(\pi_S\) is the inclusion.
\item \(T\cap H_i\) generates \(H_i\).
\item \({H_i}\) is hyperbolically embedded in \((\MCG, T)\). 
\end{rules}
\end{Def}

We remark that, by the definition of hierarchically hyperbolic group, it is always possible to find a set \(T\) such that the first two conditions are satisfied. The focus of the definition is that the family \(\{H_i\}\) is hyperbolically embedded with respect to that particular set \(T\). 

We recall some properties of hyperbolically embedded subgroups.
\begin{lemma}[\cite{AntolinSistoEndomorphismsAcyl}, Lemma 3.1]\label{lem: hyp emb -> QI emb}
If \(\{H_i\} {\hookrightarrow}_{h} (G,T)\) is a hyperbolically embedded family of subgroups, then for each \(i\) there is a finite set of generators \(Q_i\) of \(H_i\) such that \((H_i, d_{Q_i})\) is quasi-isometrically embedded in \((G, d_{T})\).
\end{lemma}

\begin{lemma}[\cite{AntolinSistoEndomorphismsAcyl}, Lemma 3.2]\label{lem: Laterals are separated in hyp emb}
Let \(\{H_i\} {\hookrightarrow}_{h} (G,T)\) be a hyperbolically embedded family of subgroups. Then for every \(\epsilon\) there  exists \(R=R(\epsilon)\) such that for every \(g \in G\) and \(i, j \in I\), if \(\mathrm{diam}(H_i \cap N_\epsilon(gH_j)) \geq R\), then \(i=j\) and \( g \in H_i\). 
\end{lemma}

\subsection{Changing the hierarchically hyperbolic structure I: plan} \label{subsec: changing str I}

The goal of this subsection is to state the main result and streamline the strategy adopted. We will postpone the bulk of the proofs to the next paragraph. 
Consider the following class of graphs of groups. 

\begin{Def}\label{def:hyperbolically_obtainable}
Let \(\GG'= \bigsqcup \GG_i\) be a finite union of finite graphs of HHG. A graph of groups \(\GG\) is \emph{star-obtainable} from \(\GG'\) if \(\GG\) can be obtained from \(\GG'\) by a finite number of the following moves.
\begin{enumerate}
\item Joining two vertices of \(\GG'\) by an edge \(e\) such that \(G_e\) is hyperbolic.
\item Adding a new vertex \(v\) and joining it with a finite set of vertices of \(\GG'\) such that 
		\begin{rules}
		\item the vertex group \(G_v\) is hyperbolic;
		\item the edge groups are quasi-convex subgroups of \(G_v\); 
		\end{rules}
\end{enumerate}
We say that \(\GG\) is \emph{hyperbolic-obtainable} from \(\GG'\) if it is star-obtainable and the following holds.
For each vertex \(v\) of \(\GG'\) the set of edge groups adjacent to \(v\) that correspond to edges of \(\GG - \GG'\) is hierarchically hyperbolically embedded in \(G_v\).
\end{Def}

We remark that the construction (1) is a special case of the second one. Indeed, adding an edge with edge group \(H_e\) is the same (on the level of the fundamental group) as adding a vertex \(v\) adjacent to edges \(e_1, e_2\) such that \(H_{e_1} \cong G_v \cong H_{e_2}\). However, we listed them separately for simplicity. 

We will now recall enough terminology to be able to state the Combination Theorem for HHG.

\begin{Def}[Graph of hierarchically hyperbolic groups, \cite{HHSII}]
Let \(\GG =(\Gamma, \{G_e\}, \{G_v\}, \phi_e^\pm)\) be a graph of groups, where \(\Gamma\) represent the underlying graph, \(\{G_V\}\) are the vertex groups, 
\(\{G_e\}\) are the edges groups and \(\{\phi_e^\pm\}\) are the edge maps. 
We say that \(\GG\) is a \emph{graph of hierarchically hyperbolic groups} if
for each $v\in V,e\in E$, we have sets $\mathfrak S_v,\mathfrak S_e$
so that $(G_v,\mathfrak S_v)$ and $(G_e,\mathfrak S_e)$ are
hierarchically hyperbolic group structures for which  $\phi_{e}^\pm \colon
G_e\to G_{e^\pm}$ is a homomorphism of hierarchically hyperbolic
groups. 
\end{Def}
Another concept that we want to recall is the one of \emph{bounded supports}. However, since the definition is long and beyond the scope of this paper, we will refer to \cite[Definition 8.5]{HHSII} and consider instead the following two Lemmas. The intuitive idea is that given two graphs of HHG that are known to have bounded supports (for instance a graph consisting of a single vertex), we can connect them in a way that the supports do not propagate from one another, yielding that the new graph still has bounded supports. The proofs are a straightforward consequence of the definition of bounded supports (\cite[Definition 8.5]{HHSII}).
\begin{lemma}\label{lem: singleton has bounded supports}
Every graph of HHG that has only one vertex and no edges has bounded supports.
\end{lemma}
\begin{lemma}\label{lem:def_bounded_support}
Let \(\GG_1\) and \(\GG_2\) be two finite graphs of hierarchically hyperbolic groups that have the bounded supports. 
Let \(\GG\) be a graph of hierarchically hyperbolic groups obtained adding edges to \(\GG_1\) and \(\GG_2\) in such a way that the result is a (connected) finite graph of hierarchically hyperbolic groups. 
Suppose that for each new edge \(e\), with edge group \((H_e, \H_e)\), the following holds. For each vertex \(v\) on \(e\), edge \(f \neq e\) incident to \(v\) and \(g_1, g_2 \in G_v\) we have that \( g_1 \phi_f^\bigcirc (\H_f) \cap g_2 \phi_e^\bigcirc (\H_e) = \emptyset\), 
where \(\phi_f : H_f \rightarrow G_v\) is the edge map.
Then \(\GG\) has bounded supports.
\end{lemma}


We recall the Combination Theorem for HHG.

\begin{thm}[Combination Theorem for HHG, \cite{HHSII} Corollary 8.22]\label{thm: combination HHG}
Let \(\GG = (\Gamma, \{G_v\}, \{G_e\}, \{\phi^\pm_e\})\) be a finite graph of hierarchically hyperbolic groups and suppose that the following are satisfied.
\begin{rules}
\item For each edge \(e\), the images  \(\phi_e^+ (G_e)\)  and \(\phi_e^- (G_e)\) have hierarchically quasi-convex image (Definition \ref{def:hierarchically_q.c.}).
\item The maps \(\phi^\pm_e\), as hieromorphisms, are full (Definition \ref{def:full}).
\item For each edge \(e\), the image of the maximal element \(S_e \in \S_e\) is not orthogonal to any element \(V \in \S_{e^\pm}\). 
\item \(\GG\) has bounded supports.
\end{rules}
Then the fundamental group of \(\GG\) admits a hierarchically hyperbolical group structure. 
\end{thm}
\begin{rmk}\label{rmk:singleton graph}
We remark that a graph with a single vertex of valence 0 satisfies the hypotheses of Theorem \ref{thm: combination HHG}. Indeed, there are no edge maps and the bounded supports follows from Lemma \ref{lem: singleton has bounded supports}.
\end{rmk}
The main result of this section is the following.

\begin{thm}\label{thm: main application combination}
Let \(\GG\) be a graph of groups that is hyperbolic-obtainable from \(\GG' = \bigsqcup \GG_i\), where each \(\GG_i\) is a finite graph of hierarchically hyperbolic groups that satisfies the hypotheses of Theorem \ref{thm: combination HHG}.
Then there are HHG structures on the vertices and edges of \(\GG\) (as in Construction \ref{construction HHG}) such that \(\GG\) satisfies the hypotheses of Theorem \ref{thm: combination HHG}. In particular, \(\pi_1 (\GG)\) admits a hierarchically hyperbolic group structure.
\end{thm}
\begin{proof}
We postpone the full proof of the Theorem to the end of this section. 
\end{proof}

This will allow to "inductively apply" the construction of Definition \ref{def:hyperbolically_obtainable} and Theorem \ref{thm: combination HHG} to obtain an HHG structure on interesting groups.
Since every hierarchically hyperbolic group can be realized as the fundamental group of a graph of HHG with one vertex and no edges, Remark \ref{rmk:singleton graph} provides an interesting class of examples. For instance, if \(H\) is a quasi-convex subgroup of a hyperbolic group \(G\) and \(H \hookrightarrow_{hh} \mathrm{MCG}\), where \(\mathrm{MCG}\) is the Mapping Class Group of a given (non sporadic) surface, we get that \(\mathrm{MCG} \ast_H G\) admits an HHG structure. 

\subsection{Changing the hierarchically hyperbolic structure II: technical part}\label{subsec: changing str II}

In this subsection we will prove Theorem \ref{thm: main application combination}. The proof will follow three steps. 
Firstly we will show that it is possible to equip the new added edge and vertex groups with an HHG structure. This is an easy consequence of Theorem \ref{thm: HHG structure on G}.
Then we will need to show that the whole graph of groups is a graph of HHG. To achieve this, we will change the HHG structure on (some) of the original vertices. This will require some amount of work.
Finally we will verify that the new graph of HHG satisfies the hypotheses of Theorem \ref{thm: combination HHG}, mainly the one concerning hierarchical quasi-convexity.

\begin{Def}[Normalized HHS]
The HHS \((\X, \S)\) is \emph{normalized} if there exists \(C\) such that for each \(U \in \S\) one has \(CU = N_C (\pi_U(\X))\).
\end{Def}

\begin{prop}[\cite{HHSBoundaries} Proposition 1.16]\label{prop:normalized HHS}
Let \(\X, \S\) be a hierarchically hyperbolic space.  Then \(\X\) admits a normalized hierarchically hyperbolic structure \((\X, \S')\) with a hieromorphism \(f \colon (\X, \S) \rightarrow (\X, \S')\) where \(f: \X \rightarrow \X\) is the identity and \(f^\bigcirc: \S \rightarrow \S'\) is a bijection.  Moreover, if
\(G \leq \mathrm{Aut}(\X, \S)\), then there is a homomorphism \(G \rightarrow \mathrm{Aut}(\S')\) making \(f\) equivariant.
\end{prop}

\begin{rmk}\label{rmk:CS=Cayley_graph}
Let \((\MCG, \S)\) be a normalized hierarchically hyperbolic group and let \(S\in \S\) be the \(\sqsubseteq\)-maximal element. Then the action of \(\MCG\) on \(CS\) is cobounded. In particular, there is a (possibly infinite) set of generators \(X\) for \(\MCG\) such that \(CS\) is quasi-isometric to \(\mathrm{Cay}(\MCG, X)\).
\end{rmk}

\begin{convention}
From now on, we will use Proposition \ref{prop:normalized HHS} to assume that all the hierarchically hyperbolic spaces are normalized. Moreover, by Remark \ref{rmk:CS=Cayley_graph}, we can assume that for each hierarchically hyperbolic group \((\MCG, \S)\), the maximal hyperbolic space \(CS\) associated it is a Cayley graph of \(\MCG\). 
\end{convention}

\begin{notation}
In what follows we will deal with different hierarchically hyperbolic structures on the same space. In order to avoid confusion, if \((\MCG, \S)\) and \((\MCG, \S')\) are two different hierarchically hyperbolic structure on \(\MCG\), we will denote by \(C_\S U\), the hyperbolic spaces associated to \(U\) in \(\S\), and by \(C_{\S'} U\) the one associated to the second structure. Similarly, we decorate the projections \(\pi_U\) and the maps \(\rho^U_V\) with left superscripts \(\prescript{\S}{}{\pi_U}\) and \(\prescript{\S}{}{\rho_V^U}\). 
\end{notation}

We want to describe how to incorporate the HHG structure of hierarchically hyperbolically embedded subgroups into an ambient HHG. 
First, consider the following construction, which is analogous to the one described in \cite[Proposition 6.14]{HHSAsdim2015}

\begin{construction}[Cf. \cite{HHSAsdim2015}]\label{construction HHG}
\emph{Let \(\{(H_i, \H_i)\}_{i \in I}\) be a finite family of hyperbolic hierarchically hyperbolic groups, and let \((\MCG, \S)\) be a hierarchically hyperbolic group. Suppose that the family \(\{H_i\}_{i \in I}\) is hierarchically hyperbolically embedded in \(\MCG\). We construct a new HHS structure on \(\MCG\).}

For each \(i \in I\),  let \(\phi_i: H_i \hookrightarrow \MCG\) be the inclusion maps.
By Remark \ref{rmk:CS=Cayley_graph}, there is a set \(T \subseteq \MCG\) so that \(C_\S S =\mathrm{Cay} (\MCG, T)\).
Let \(\mathrm{Cos}(\{H_i\})\) be the set of all the cosets of the various \(H_i\) in \(\MCG\), and let \(\Pyr\) be the cone-off of \(CS\) with respect to \(\mathrm{Cos}(\{H_i\})\).

By Lemma \ref{lem: hyp emb -> QI emb}, the maps \(\phi_i: H_i \hookrightarrow CS\) provide (uniform) quasi-isometric embeddings for all the elements of \(\mathrm{Cos}(\{H_i\})\).
Since each coset \(g\phi_i(H_i)\) is isometric to \(\phi_i(H_i)\), since HHS structures are preserved by quasi-isometry, and since, by assumption, \((H_i, \H_i)\) is a hierarchically hyperbolic group, we can push forward the HHS structure of \(H_i\) to all the cosets \(g\phi_i (H_i)\).

The new HHS on the group \(\MCG\) will be obtained "adding" all the above HHS structures to \((\MCG, \S)\). Since all the HHS structures on the various cosets (of the same \(H_i\)) are identical, we will describe a precise indexing set for all the HHS structures of the elements of  \(\mathrm{Cos}(\{H_i\})\).
For each \(i \in \I\) let \(\MCG (i)\) be a maximal set of elements of \(\MCG\) such that for each pair \(g_1, g_2 \in \MCG (i)\), we have that \(g_1 g_2^{-1} \not \in H_i\) (that is, choose once and for all a representative for the cosets of \(H_i\)). 
Then for each \(i \in I\) define \(\Z_i\) to be the set \(\H_i \times \MCG (i)\), and let \(\Z = \bigsqcup_{i \in I} \Z_i\). 
 
We claim (Proposition \ref{prop:properties_of_the_new_structure}) that the following defines a hierarchically hyperbolic group structure on \(\MCG\):
\begin{itemize}
\item The index set is \(\I= \S \sqcup \Z\)
\item The relations \(\sqsubseteq\) and \(\bot\) restricted to \(\S\) are unchanged. Between the elements of a level set \(\H_i \times \{g\}\subseteq \Z\), the relations are defined as on \(\H_i\), and we set all the elements of \(\Z\) to be nested into \(S\). In all the other cases, we declare two elements to be transverse (that is, neither orthogonal nor nested). 
\item The spaces \(C_\I W\) are defined as: 
	\begin{rules}
	\item \(\Pyr\) if \(W= S\); 
	\item \(C_\S W\) if \(W\in \S -\{S\}\); 
	\item \(C_{\H_i}U \) if \(W =(U, g)\), where \(U \in \H_i\) and \(g \in \MCG(i)\).
	\end{rules}
\item The projections \(\prescript{\I}{}{\pi_W} : \MCG \rightarrow 2^{C_\I W}\) are unchanged on \(\S\). Given an element \(W \in \Z\), there is \(i \in I\), \(U \in \H_i\) and \(g \in \MCG (i)\) such that \(W = (U, g)\). Then we define \(\prescript{\I}{}{\pi_W}\) as the composition \(\prescript{\H_i}{}{\pi_U} \circ p_{gH_i}\), where \(p_{H_i}\) denotes the closest point projection in  \(CS\) on the coset \(gH_i\), and we identify \(gH_i\) with \(H_i\).
\item The maps and spaces \(\prescript{\I}{}{\rho^U_V}\) are:
	\begin{rules}
	\item unchanged if \(U\) and \(V\) are elements of \(\S\); 
	\item unchanged if \(U\) and \(V\) are elements of \(\H_i\times \{g\}\) for some \(i \in I\) and \(g \in \MCG(i)\);
	\item defined as follow if we are not in the above cases. Let \(W \in \{U,V\}\) and define \(S_W\) as \(S\) if \(W \in \S\), or as \(gH_i\) if 
	\(W \in Z_i\). We have that \(S_U\) and \(S_V\) can be  uniquely seen as subgraphs of \(CS\), so it is well defined the closest point projection 
	\(p_{S_V}\), where we remark that we consider the metric in \(CS\).
	Then we set \(\rho^U_V\) as the composition \(\pi_V \circ p_{S_V} \circ \rho^U_{S_U}\), where \( \rho^U_{S_U}\) is defined 
	as in the lines above and \(\pi_V\) is well defined, since by hypotheses \(CS\) is a Cayley graph for \(\MCG\).
	\end{rules}	
\end{itemize}

\end{construction}

\begin{prop}\label{prop:properties_of_the_new_structure}
The structure described in Construction \ref{construction HHG} is a hierarchically hyperbolic structure on \(\MCG\) such that:
\begin{rules}
\item \((\MCG, \I)\) is a hierarchically hyperbolic group;
\item \(\S \subseteq \I\); 
\item for each \(i \in I\), the map \(\phi_i: H_i \rightarrow \MCG\) is a homomorphism of hierarchically hyperbolic groups, with respect to the above structures.
\item for each \(i \in I\) we have that \(\phi_i^\bigcirc (\H_i) \cap \S = \emptyset\), where \(\phi_i^\bigcirc: \H_i \rightarrow \I\) is the map on the index set level induced by \(\phi\).
\end{rules}
\end{prop}
\begin{proof}
The second and last point are clear by construction. The hard part of the proof consists in showing that \((\MCG, \I)\) is an HHS, and we will do by verifying the axioms of hierarchically hyperbolic space. The fact that \((\MCG, \I)\) is an HHG and that the maps \(\phi_i\) are homomorphisms of hierarchically hyperbolic groups follow from the fact the HHS structures  are clearly compatible with the \(\MCG\) and \(H_i\) actions.

\emph{\(\bullet\) There is a set of projections \(\{\pi_W \mid  W \in \I\}\) that are uniformly quasi-Lipschitz.}
This follows because the maps \(\prescript{\Z}{}{\pi_W}\) are defined as composition of uniformly quasi-Lipschitz maps.

\emph{\(\bullet\) There is a relation \(\sqsubseteq\) and sets \(\rho^U_V\) such that for each pair \(U, V\) with \(U\) properly nested in \(V\), the set \(\rho^U_V\) has uniformly bounded diameter}. Lemma \ref{lem: Laterals are separated in hyp emb} and the fact that \((H_i, \H_i)\) is a hierarchically hyperbolic group gives the statement on \(\Z\). Since the map \(CS \rightarrow \Pyr\) is distance-non-increasing, it holds also on \(\S\). Since, by construction, if two elements are nested, then they both belong to either \(\S\) or \(\Z\) or one is in \(\Z\) and the other is \(S\). But since each element of \(\Z\) is coned-off in \(\Pyr\), we get the claim.

\emph{\(\bullet\) Orthogonality.} This clearly holds because if \(U \bot V\), then \(U, V \in \S\).

\emph{\(\bullet\) Behrstock's inequalities.}
	We start with an useful lemma.
	
	\begin{lemma}\label{lem: The element of H projects uniformly on the U}
	Let \((\MCG, \S)\) be a hierarchically hyperbolic group, \(S \in \S\) be the maximal element and \(X\) be a set of generators for \(\MCG\) such that \(CS = \mathrm{Cay}(\MCG, X)\) (as in Remark \ref{rmk:CS=Cayley_graph}). Let \(H\) be a finitely generated group and \(Y\) a finite set of generators for \(H\). Finally let \( \phi \colon H \rightarrow \MCG\) be a group homomorphism that is a quasi-isometric embedding with respect to the word metrics in \(Y\) and \(X\).
Then there is \(C\) such that for each \(g \in \MCG\), and \(U \in \S -S\), one has \(\mathrm{diam}_{CU}(\pi_U (g\phi(H))) \leq C\).
\begin{proof}
We recall that \(CS= \mathrm{Cay}(\MCG, X)\). 
Since, for each \(U \in \S -\{S\}\), the space \(\rho^U_S\) has uniformly bounded diameter in \(CS\), and since \(H\) is quasi-isometrically embedded in \(CS\), the following holds.
There is a uniform radius \(\Delta = \Delta (H)\) and for each \(U \in \S -\{S\}\) a point \(x_U \in CS\) such that 
\begin{rules}
\item \(N_E (\rho^U_S) \subseteq B_\Delta (x_U)\), where \(E\) is the constant for the bounded geodesic image property for \(\S\), 
\item For each \(g \in \MCG\) and \(y \in gH - B_\Delta (x_U)\), the projection of \(y\) on \(B_\Delta (x_U) \cap gH\) is not contained in \(N_E(\rho^U_S)\) (see figure \ref{fig:_projection on the outer ball}).
\end{rules}

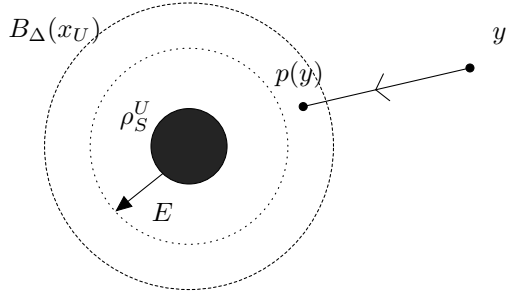
\begin{figure}[h]
\begin{center}
\begin{tikzpicture}[line cap=round,line join=round,>=triangle 45,x=1cm,y=1cm]
\draw [fill=black,fill opacity=0.86] (0,0) circle (0.5cm);\draw [dotted] (0,0) circle (1.3cm);\draw [dash pattern=on 1pt off 1pt] (0,0) circle (1.9cm);\draw [->] (-0.3448666039168691,-0.3620318017838286) -- (-0.9698946691989615,-0.8656236657228342);\draw (3.6946765546447127,1.0382832420117576)-- (1.501609944546866,0.5265676996555935);\draw (2.455763340220989,0.7492034919795559) -- (2.5566157760281407,0.9604003575521776);\draw (2.455763340220989,0.7492034919795559) -- (2.639670723163441,0.6044505841151749);
	\draw[color=black] (-0.7,0.4) node {$\rho^U_S$};
	\draw[color=black] (-1.7636225638210388,1.5378150809784894) node {$B_\Delta(x_U)$};
	\draw [fill=black] (3.6946765546447127,1.0382832420117576) circle (1.5pt);
	\draw[color=black] (4.084555063106552,1.4647128606418947) node {$y$};
	\draw [fill=black] (1.501609944546866,0.5265676996555935) circle (1.5pt);
	\draw[color=black] (1.4528751309891361,0.9529973182857303) node {$p(y)$};
	\draw[color=black] (-0.35031297064687117,-0.8745581901291419) node {$E$};
\end{tikzpicture}
\end{center}
\caption{The projection of the point \(x\) on the set \(B^U_\Delta\) lies outside the \(E\)-neighborhood of \(\rho^U_S\). }\label{fig:_projection on the outer ball}
\end{figure}
Firstly, we want to uniformly bound the diameter of the various projections \(\pi_U(gH \cap B_\Delta (x_U))\). Note that the projections \(\pi_U \colon \MCG \rightarrow 2^{CU}\) are uniformly quasi-Lipschitz, but with respect to the metric on \(\MCG\) which is, in general, not quasi-isometric to \(CS\). We will show that is nevertheless possible to uniformly bound the diameter in \(\MCG\) of the sets \(gH \cap B_\Delta (x_U)\), which gives the desired estimate. 

Let \(F = B_{2\Delta} (\1) \cap H\).
Since \(H\) is finitely generated and quasi-isometrically embedded in \(CS\), \(F\) is a finite set. Moreover, for each \(g \in \MCG\) and \(U \in \S -\{S\}\), there exists \(h \in gH\) such that \(gH \cap B_\Delta (x_U) \subseteq hF\). For instance, choosing \(h\) to be the closest point projection of \(x_U\) on \(gH\) works. 
Since the set \(F\) is finite, it has finite diameter in \(\MCG\). Moreover, left multiplication by \(h\) is an isometry. Thus we can uniformly bound the diameter in \(\MCG\) of of the sets \(hF\), and thus of the sets \(gH \cap B_\Delta (x_U)\).

Now, consider a point \(z \in gH - B_\Delta (x_U)\). We claim that the projection \(\pi_U(z)\) lies in a uniform neighborhood of \(\pi_U(gF)\). In fact, since \(z \not\in B_\Delta (x_U)\),  \(z\) lies outside the ball \(B_\Delta (x_U)\) in \(CS\). Let \(\gamma\) be a geodesic of \(CS\) between \(z\) and its closest point projection on \(B_\Delta(x_U)\). By construction \(\gamma \cap N_E(\rho^U_S) = \emptyset\), and thus, by the bounded geodesic image property, there is a uniform \(B\) such that \(\mathrm{diam}_{CU}(\pi_U(\gamma)) \leq B\). In particular, \(\pi_U(z)\) lies in the \(B\)-neighborhood of \(\pi_U(B_\Delta (x_U))\), which concludes the proof.
\end{proof}
\end{lemma}

%

We will now proceed to prove Behrstock's inequalities.

	\begin{enumerate}
	\item[1] We have to show that there exists a uniform \(\kappa\) such that for all \(U, V \in \I\) not nested into each other, it holds: 
	\[\mathrm{min}\{d_U(\pi_U(x), \rho^V_U), d_V(\pi_V(x), \rho^U_V)\} \leq \kappa.\]
	It is easily seen that there are only two cases we should worry about: when \(U \in \Z_i\) and \(V \in \Z_j\) with \(i \neq j\), or when
	\(U \in \S-\{S\}\) and \(V \in \Z-\{S\}\).
	Consider the first one. Unraveling the definitions, it is clear that it suffices to show the case when \(U=g_1H_i\) and \(V= g_2H_j\), with 
	\(i \neq j\). However, Lemma \ref{lem: Laterals are separated in hyp emb} and Corollary \ref{cor: Simple Factor System} give that the set of the various laterals of the \(H_i\) in \(\MCG\)
	 is a factor system for \(CS\). Then the result follows. 
	 	
	So, consider the second case. Let \(i \in I\), \(g \in \MCG (i)\) and \(J \in \H_i\) such that \(V= (J, g)\).
	Unraveling the definitions, we get that we need to show that there exists a uniform \(\kappa\) such that:
	\[\mathrm{min}\{d_U(\pi_U(x), \pi_U(\rho^V_S)), d_{C_{\H_i} J}(\pi_J \circ p_{gH_i}(x),\pi_J\circ p_{gH_i}(\rho^U_S))\} \leq \kappa,\]
	were we remark that all the projections are considered with the metric of \(CS\). 
	We will show that if the second term is sufficiently large, then we can uniformly bound the first one.
	Note that, since the maps \(\pi_J \) are uniformly quasi-Lipschitz and since \(gH_i\) is quasi-isometrically embedded in \(CS\),
	 we can substitute the second term with 
	\(d_{CS}(p_{gH_i}(x),p_{gH_i}(\rho^U_S))\).
	Let \(E, \kappa_0\) be the constants coming from the bounded geodesic image and Behrstock's inequality conditions of the HHG \((\MCG, \S)\). 
	Let \(\Delta = E + \kappa_0\). Let \(y=p_{gH_i}(x)\) and \(z= p_{gH_i}(\rho^U_S)\). 
	Let \(\gamma\) be a geodesic of \(CS\) between \(x\) and \(y\). 
	Since the closest point projection on \(gH_i\) is a quasi-Lipschitz map, say with constant \(L\), if \(d_{CS}(y,z) > L\Delta +L\), 
	we obtain that \(\gamma \cap N_E(\rho^U_S) = \emptyset\). 
	The bounded geodesic image condition applied  to \(\gamma\), 
	gives that \(\mathrm{diam}_{CU}(\rho^S_U(\gamma))\) is uniformly bounded.
	In particular, since \(y \in {gH_i}\), we get that \(d_U(\rho^S_U(x), \rho^S_U({gH_i}))\) is uniformly bounded.
	Since \(U, S \in \S\), the Behrstock's inequality condition gives that, for every \(x \in \MCG\), we have
	\[\mathrm{min}\{d_{CS}(x, \rho^U_S), \mathrm{diam}_{CU}(\pi_U(x) \cup \rho^S_U(x))\}\leq \kappa_0.\]
	But we have that \(d_{CS} (\{x,y\}, \rho^U_S) > L\Delta +L \geq \kappa_0\). Thus, we must have \(d_{CU}(\pi_U(x), \pi_U({gH_i}))\) is uniformly bounded. 
	Since, by Lemma \ref{lem: The element of H projects uniformly on the U}, the projection \(\pi_U(gH_i)\) is uniformly bounded, we get the claim.
	\item[2] We have to show that for each \(V \sqsubseteq W \) properly nested into each other, there is an uniform \(\kappa\) such that
	\[\mathrm{min}\{d_{W}(x, \rho^V_W), \mathrm{diam}_{CV}(\pi_V(x) \cup \rho^W_V(x))\}\leq \kappa_0.\]
	The only case that we need to check is when \(W = S\). 
	So, suppose that \(V = (J, g) \in \Z\), where \(J \in \H_i\), and \(g \in \MCG(i)\). 
	Then, unraveling the definitions, the second term becomes \(\mathrm{diam}_{CJ}(\pi_J \circ p_{gH_i}(x) \cup \pi_J \circ p_{gH_i}(x))\), which has 
	uniformly bounded diameter since \(\pi_J \circ p_{gH_i}\) is uniformly quasi-Lipschitz.
	
	Suppose, instead, that \(V \in \S -\{S\}\).
	Since distance in \(\Pyr\) are smaller or equal distances in \(CS\), and since the above statement holds in \(CS\), we get the claim.
	\item[3] We need to show that whenever \(U \sqsubseteq V\) and there is \(W\) which is not orthogonal or nested into \(V\), it holds 
	\[d_W(\rho^U_W, \rho^V_W) \leq \kappa.\]
	Unraveling the definitions, it is easily to see that if \(U,V,W  \in \Z\), then the claim holds, and since distances in \(\Pyr\) are smaller or equal 
	distances in \(CS\), it also holds if 
	\(U, V, W \in \S\). So, there are only two cases of interest: \(U, V \in \S-\{S\}\) and \(W \in \Z -\{S\}\), or \(U, V \in \Z-\{S\}\) and \(W \in 
	\S -\{S\}\). 
	The second case is clear because, unraveling the definitions, one has \(\rho^U_W \subseteq \rho^V_W\).
	The first one follows from the fact that the statement holds in \(CS\) and the projections \(\rho^S_U\) and \(\rho^S_V\) are uniformly quasi-Lipschitz.
	\end{enumerate}

\emph{\(\bullet\) Finite complexity.} This is clear by construction.

\emph{\(\bullet\) Large link lemma: there exists \(\lambda \geq 1\) and \(E \) such that the following holds.
Let \(W \in \I\) and \(x,y \in \MCG\). Let \(N = \lambda d_{W}(\pi_W(x), \pi_W(y)) + \lambda\).
Then there exists \(\{T_1, \dots , T_{\lceil N \rceil}\} \subseteq \H_W\) such that for each \(V \in \H_W\), either \(V \in \H_{T_i} \cup \{T_i\}\) for some \(i \in \{1, \dots, \lceil N \rceil\}\), or \(d_V(\pi_V(x), \pi_V(y)) < E\). Also, \(d_{W} (\pi_W(x), \rho^{T_i}_W) \leq N\) for each i.}\\

Note that if \(W \in \Z\) or \(W \in \S - S\), then the claim holds because of the HHS structure of \((\MCG, \S)\) or \((H_i, \H_i)\), for the various
\(i\). 
So, we are only interested in the case \(W= S\).
We need to show that given \(x,y \in \MCG\), it is possible to linearly bound in terms of the distance in \(\Pyr\) the number of (\(\sqsubseteq\)-maximal) elements \(U\) of \(\I\) with the property that \(d_{CU}(\pi_U(x), \pi_U(y)) \leq E\). We will consider separately the elements in \(\Z\) and in \(\S\). 

Consider the set \(A\) consisting of all \(\MCG\)-translates of the various \(H_i\). We will use the following two important facts.
By Corollary \ref{cor: Simple Factor System} and Lemma \ref{lem: Laterals are separated in hyp emb}, we can easily see that \(A\) is a factor system for \(CS\). Since, by definition, \(\Pyr\) is the cone-off of \(CS\) with respect to \(A\), the Large Link Lemma holds for the set \(A\). 

Secondly, since each \(gH_i\) is uniformly quasi-isometrically embedded in \(CS\), and since the projections \(\pi_U\), for \(U \sqsubseteq gH_i\), are uniformly quasi-Lipschitz, we get that there exists a non-decreasing map \(L: \RR \rightarrow \RR\) such that the following holds. 
For each \(i \in I, g \in \MCG(i)\) and  \(U \sqsubseteq gH_i\) we have that \(d_{CU}\left(\pi_{CU}(x), \pi_{CU}(y)\right) \leq L\left(d_{gH_i}\left(\pi_{gH_i} (x), \pi_{gH_i}(y)\right)\right)\), for each \(x,y \in \MCG\).

Let \(E_1\) be the constant of the Large Link Lemma applied to the set \(A\) and let \(E_2 > L(E_1)\).
The Large Link Lemma on the set \(A\) give us a constant \(\lambda_A\) and a set \(\T^A\), with \(|\T^A| \leq \lambda_A d_{\Pyr} (x,y) + \lambda_A\), such that \(gH_i \in \T^A\) if and only if \(d_{gH_i} (p_{gH_i} (x), p_{gH_i}(y)) \leq E_1\). But then, for each \(U \in \Z\), we have that \(d_{CU}(\pi_U(x), \pi_U(y)) \leq E_2\) if and only if \(U \sqsubseteq gH_i\) with \(d_{gH_i}(\pi_{gH_i}(x), \pi_{gH_i}(y)) > E_1\). Thus, for each \(E_2 > L(E_1)\) the Large Link Lemma holds on the set \(\Z\).

Note that the key ingredient in the above proof is that we already have a version of the Large Link Lemma that holds with the distances in \(\Pyr\). 
 If such a result was available also for \(\S\), then, since each element of \(\I\) is contained in either \(\Z\) or \(\S\), we would get the result. However, given two points \(x,y \in \Pyr\), it is, in general, impossible to estimate \(d_{CS} (x,y)\) in terms of \(d_{\Pyr}(x,y)\).

We will show that, for a large enough \(E\), it is however possible to estimate the number of (maximal) elements  \(W \in \S\) with \(d_W(\pi_W(x), \pi_W(y))> E\) in terms of \(d_{\Pyr}(x,y)\), which is enough for our purposes. This will require some amount of work.

\begin{lemma}\label{lem: gamma*beta*gamma}
For each \(C>0\)  there exist \(T=T(CS, C)\) and \(\Theta=\Theta (CS)\) such that the following holds.
Let \(\gamma= \gamma_0*\beta_1*\gamma_1* \cdots *\beta_M*\gamma_{M}\) be a \(C\)-quasi-geodesic of \(CS\) such that for each \(i\) there is \(g_i\) and \(j \in I\) such that \(\beta_i \subseteq g_i H_j\) and \(L(\beta_i)> \Theta\).
Then for each \(U \in \S -\{S\}\), there exists \(0\leq i \leq M\) such that 
\(\mathrm{diam}(\pi_U(\gamma)) - T \leq \mathrm{diam}(\pi_U(\gamma_i))\).
\end{lemma}

\begin{proof}
Let \(E\) be the constant of the bounded geodesic image property for \(\S\) and \(H\) be such that all \(C\)-quasi-geodesics of \(CS\) with the same endpoints have Hausdorff distance at most \(H\). Choose \(U \in \S -\{S\}\) and  consider \(B=  N_{E+H} (\rho^U_S)\). If \(\gamma\) does not intersect \(B\) the result trivially follows from the bounded geodesic image for \(\S\). So suppose that \(\gamma\) intersects \(B\). Roughly speaking, if \(\Theta\) is large enough, all the segments \(\gamma_i\) of \(\gamma\) becomes "far away". In particular, we will get that there is at most one segment \(\gamma_i\) that is near \(B\). More precisely, setting \(\Theta = C(3H +E) +C\) we obtain that there is \( i \in \{0, \dots, M\}\) such that \((\gamma - (\beta_i * \gamma_i * \beta_{i+1})) \cap B = \emptyset\). Let \(\chi_1, \chi_2\) be the connected components of \(\gamma - (\beta_i * \gamma_i * \beta_{i+1})\).  The bounded geodesic image for \(\S\)  gives that \(\mathrm{diam}(\pi_U(\chi_j))\) is uniformly bounded, for \(j \in \{1,2\}\). Since each of the segments \(\beta_i, \beta_{i+1}\) is contained in some element of \(\mathrm{Cos}(\{H_i\})\), by Lemma \ref{lem: The element of H projects uniformly on the U} we get the claim.
\end{proof}

So, consider points \(x, y  \in \Pyr\). Since \(\Pyr\) is obtained from \(CS\) coning-off a family of uniformly quasi-isometrically embedded subgraphs, we can apply Proposition \ref{prop: Bound Hausdorff distance}  and get a quasi-geodesic \(\widehat{\gamma}\) of \(\Pyr\) with the property that each embedded-de-electrification \(\widetilde{\gamma}^e\) of \(\widehat{\gamma}\) is uniformly a quasi-geodesic of \(CS\). Fix such a de-electrification \(\widetilde{\gamma}^e\). 

We want to apply Lemma \ref{lem: gamma*beta*gamma} to \(\widetilde{\gamma}^e\). In order to do that, we will subdivide \(\widetilde{\gamma}^e\) in a way that satisfies the hypotheses of Lemma \ref{lem: gamma*beta*gamma}.

Recall that the quasi-geodesic \(\widehat{\gamma}\) decomposes as 
\[\widehat{\gamma} = \widehat{\gamma}_0 \ast e_1 \ast \cdots \ast e_M \ast \widehat{\gamma}_M,\] where each \(e_i\) is a \(\mathrm{Cos}(\{H_i\})\)-component and all the \(\widehat{\gamma}_i\) are quasi-geodesics of \(CS\) (in particular, they do not contain \(\mathrm{Cos}(\{H_i\})\)-components).
Thus, \(\widetilde{\gamma}^e\) has the form \[\widetilde{\gamma}^e = \widehat{\gamma}_0 \ast \beta_1 \ast \cdots \ast \beta_M \ast \widehat{\gamma}_M,\] where each \(\beta_i\) is obtained substituting each \(e_i \in g_i H_j\) with a geodesic of \(g_i H_j\) connecting the endpoints of \(e_i\).
 Let \(\Theta \) be the constant coming from Lemma \ref{lem: gamma*beta*gamma}. We will choose a coarser subdivision of \(\widehat{\gamma}\) and \(\widetilde{\gamma}^e\) such that all the subsegments \(\beta_j\) have length greater or equal to \(\Theta\). More formally, we obtain a new subdivision as follows: let \(j \in \{1, \dots ,  M\}\) be the first index such that \(L(\beta_j) < \Theta\). 
Then for each \(i < j-1\), set \(\gamma_i' = \widehat{\gamma}_i\) and \(\beta_i' = \beta_i\). Set \(\gamma_{i-1}'= \widehat{\gamma}_{i-1} * \beta_i *\widehat{\gamma}_i\) and for each \(i \geq j\) set \(\gamma_i'=\widehat{\gamma}_{i+1}\) and \(\beta_i' = \beta_{i+1}\). Proceeding in this way, we obtain a new subdivision:
\[\widetilde{\gamma}^e = \gamma_0 \ast \beta_1 \ast \cdots \ast \beta_n \ast \gamma_n,\] 
where each of the \(\beta_i\) is contained in one element of \(\mathrm{Cos}(\{H_i\})\) and has length greater or equal \(\Theta\), and each of the \(\gamma_i\) may contain some \(\mathrm{Cos}(\{H_i\}\))-piece of length strictly less than \(\Theta\). 

For each segment \(\gamma_i\), let \(\overline{\gamma_i}\) be the subsegment of \(\widehat{\gamma}\) with the same endpoints of \(\gamma_i\) (that is, the segment obtained substituting back the \(\mathrm{Cos}(\{H_i\})\)-components in \(\gamma_i\)).
It is clear that \(\gamma_i\) is an embedded-de-electrification of \(\overline{\gamma_i}\), but with the remarkable property that each \(\mathrm{Cos}(\{H_i\})\)-components \(e\) of \(\overline{\gamma_i}\) de-electrify to a \(CS\) segment of length at most \(\Theta\).
Thus, for each \(i\) the following holds: 
\[L_{CS} (\gamma_i) \leq \Theta L_{\Pyr} (\overline{\gamma_i}).\]
Let \(\tau\) be such that \(\widehat{\gamma}\) is a \(\tau\)-quasi-geodesic of \(\Pyr\). We have: 
\[\sum_0^n L_{CS} (\gamma_i) \leq \sum_0^n \Theta L_{\Pyr}(\overline {\gamma_i}) \leq \Theta L_{\Pyr}(\widehat{\gamma}) \leq \tau \Theta d_{\Pyr}(x,y) + \tau.\]

That is, we can estimate the length of the portion of \(\widetilde{\gamma}^e\) that is coarsely not contained in \(\mathrm{Cos}(\{H_i\})\) in terms of \(d_{\Pyr}(x,y)\). 

Since \(\widetilde{\gamma}^e\) is uniformly a quasi-geodesic of \(CS\),    \ref{lem: gamma*beta*gamma} gives that there is a uniform \(F\) such that for each \(U \in \S -\{S\}\), if \(\mathrm{diam}_{CU}(\pi_U (\{x,y\})) > F\), then  there is \(i\) such that \(\mathrm{diam}_{CU}(\gamma_i)> E_\S\), where \(E_\S\) is the large link lemma constant for \(\S\). 
The large link lemma property for \(\S\) furnish a series of subsets \(\T_i \subseteq \S-\{S\}\) such that if there exists \( U \in \S -\{S\}\) with \(\mathrm{diam}_{CU}(\gamma_i)> E_\S\), then \(U\) is nested in one element contained in \(\T_i\). 
Let \(\T = \cup_i \T_i\). We have that there is a uniform \(\lambda\) such that \(|\T| \leq \lambda (\sum_{i=0}^n L_{CS} (\gamma_i)) +\lambda\). 
Since the latter can be uniformly linearly estimated in terms of \(d_{\Pyr} (x,y)\), we get the result.

\emph{\(\bullet\) Bounded geodesic image: there exists \(E_\I\) such that for all \(W \in \I\), \(V \in \I_W -\{W\}\) and geodesics \(\gamma\) of \(CW\), either \(\mathrm{diam}_{CV} (\rho^W_V(\gamma)) \leq E_\I\), or \(N_{E_\I} (\rho^V_W) \cap \gamma = \emptyset\).}

The statements clearly holds in the case \(W \neq S\) or \(W = S\) and \(V \in \Z\). So, we only have to check the statement for \(W = S\) and \(V \in \S\). 
But since, by Proposition \ref{Kapovich-Rafi},  geodesic of \(CS\) and geodesics of \(\Pyr\) have uniformly bounded Hausdorff distance (with respect to the metric in \(\Pyr\)), the conclusion follows from the bounded geodesic image property for \(CS\).

\emph{\(\bullet\) Partial Realization}
Note that a family of pairwise orthogonal elements is either contained in \(\S\) or \(\Z\). Since \(\Z\) is a factor system, it holds in \(\Z\), and ii holds on \(\S -\{S\}\). Thus we get that the property holds on the whole \(\I\).

\emph{\(\bullet\) Uniqueness}
This holds in \(\Z\), and that is enough, since for each \(W \in \Z\), \(C_\I W = C_\Z W\) and \(\prescript{\I}{}{\pi_W} = \prescript{\Z}{}{\pi_W}\). 

\end{proof}

We are now ready to prove Theorem \ref{thm: main application combination}.

\begin{proof}[Proof of Theorem \ref{thm: main application combination}]
The proof will follow three steps. 
Firstly we will show that it is possible to equip the new added edge and vertex groups with an HHG structure. 
Then we will need to show that the whole graph of groups is a graph of HHG. To achieve this, we will change the HHG structure on (some) of the original vertices. 
Finally we will verify that the new graph of HHG satisfies the hypotheses of Theorem \ref{thm: combination HHG}. 

For each vertex \(v \in \GG - \GG'\), we have that \(G_v\) is hyperbolic and that the adjacent edge groups \(\{H_i\}\) are a finite family of quasi convex subgroups of \(G_v\). Then Theorem \ref{thm: HHG structure on G} applied to \((G, \{H_i\})\) yields a hierarchically hyperbolic group structure  \((G, \F)\) on \(G\). This induces a hierarchically hyperbolic structure \((H_i, \H_i)\) for each \(i\),  such that the inclusion maps \(j_i: H_i \hookrightarrow G_v\) are homomorphisms of hierarchically hyperbolic groups. 
Indeed, this follows because \([H_i] \in \F\), \(\left( C_\F H_i, \F_{H_i} \right)\) admits an HHS structure (see Corollary \ref{cor: sub-fac-sys Gps}) and  \(H_i\) is quasi-isometric to \(C_\F H_i\).
It is straightforward that all the maps \(j_i\) are full and, since there are no orthogonality relations in \(\F\), the images of the maximal elements are not orthogonal to any other element. 

For each vertex \(v \in \GG'\), let \(\{H_i^v\}\) be the set of edge groups adjacent to \(v\) that correspond to edges of \(\GG - \GG'\). By hypothesis, we have that \(\{H_i^v\} \hookrightarrow_{hh} G_v\). Note that, by the process above, all elements of \(\{H_i^v\}\) are equipped with an HHG structure.  Applying Proposition \ref{prop:properties_of_the_new_structure} to the pair \((G_v, \S_v); \{(H_i^v, \H_i^v)\}\), we get new HHG structures \((G_v, \I_v)\) on the \(G_v\) that turn \(\GG\) into a graph of hierarchically hyperbolic groups. 
The bounded supports condition follows from Lemma \ref{lem:def_bounded_support}, the fullness of the maps and the orthogonality condition are clear by construction.

We need to show that the images of the edge maps are hierarchically quasi-convex. To simplify notation, we will denote \(d_{\widehat{V}}(\pi_V(x), \pi_V(y))\) as \(d_{\widehat{V}}(x,y)\).
Consider an edge \(e\) and a vertex \(v\) adjacent to it. Let \(\Y\subseteq \X\) be the image of the edge map. We want to show that \(\Y\) is hierarchically quasi-convex. There are several cases that needs to be considered.

\textit{The case where \(e\) is an edge of \(\GG'\).}
If the hierarchical structure on the vertex group \(\MCG_v\) is unchanged, then \(\Y\) is hierarchically quasi-convex by hypothesis. 
Otherwise, let \((\MCG_v, \S_v)\) be the old hierarchical structure, \((\MCG_v, \I_v)\) be the new one and \(S\) be the maximal element of both \(\S_v\) and \(\I_v\). Finally, let \(CS\) denote the space \(C_{\S_v}S\) and let \(\Pyr\) denote the space \(C_{\I_v}S\).  

Since \(\Y\) is hierarchically quasi-convex in \((\MCG_v, \S_v)\) (and in particular quasi-convex), we get that for each element \(gH_i \in \mathrm{Cos}(\{H_i\})\), the projection \(p_{gH_i}(\Y)\) is uniformly quasi-convex.
Let \(x \in \MCG_v\) be such that \(d_{CU}(x,\Y)\leq r\) for all \(U \in \I_v\).
We want to show that it is possible to uniformly bound \(d_{\MCG}(\Y, x)\).
 If we can bound \(d_{CS}(x,\Y)\), then the result would follows from the fact that \(\Y\) is hierarchically quasi-convex in \((\MCG_v, \S_v)\). In order to obtain this bound we will consider two additional different sets of HHS structures on the various spaces involved.
  
Firstly, let \(A\) be the set of cosets of the various \(H_i\) seen as subspaces of \(CS\). By Corollary \ref{cor: Simple Factor System} and Lemma \ref{lem: Laterals are separated in hyp emb}, we have that \(A\) is a factor system for \(CS\). 
By Proposition \ref{prop: Y is hierarchically quasi-convex}, we have that \(\Y\) is hierarchically quasi-convex in \((CS, A\cup\{CS\})\). Note that \(\Pyr\) is quasi-isometric to the cone-off of \(CS\) with respect to \(A\), and, by hypothesis, we have \(d_{\Pyr}(\Y, x) \leq r\). 
If for each element \(gH_i \in A\), we could uniformly bound \(d_{CgH_i}(\Y, z)\), then hierarchically quasi-convexity of \(\Y\) in \((CS, A \cup \{CS\})\) would produce a bound on \(d_{CS}(\Y, x)\) and thus, a bound on \(d_{\MCG}(\Y, x)\). 

Thus, consider a coset \(gH_i\). If \(W=(H_i, g) \in \Z\), one has that \((gH_i, \Z_W)\) is a hierarchically hyperbolic space, with hyperbolic total spaces. Thus we can apply Proposition \ref{prop: Y is hierarchically quasi-convex} to obtain that \(\Y\) is hierarchically quasi-convex in \((gH_i, \Z_W)\). Since, by hypothesis, for each \(U \in \Z_W (\subseteq \I_v)\) we have \(d_{CU}(\Y, x)\leq r\), there is a function \(k\) depending on \(i\) but not on \(g\) such that  \(d_{gH_i}(x,\Y)\leq k(r)\). Since there are only finitely many choices for \(i\), the result follows.

\textit{The case where \(e\) is an edge of \(\GG -\GG'\) and \(v\) is a vertex of \(\GG'\).}
The image of the edge map \(\Y\) coincide with one of the subgroups \(H_i\). 
Then we have that \(\pi_U(\Y)\) is uniformly bounded when \(U \in \I - \Z\), because of Lemma \ref{lem: The element of H projects uniformly on the U}, or when \(U \sqsubseteq (g, H_j)\), with \(j \neq i\), because of Lemma \ref{lem: Laterals are separated in hyp emb}.
In the other case, \(\pi_U(\Y)\) coarsely coincide with \(CU\) by normalness (Proposition \ref{prop:normalized HHS}).
This clearly implies that for each \(U\), the space \(\pi_U(\Y)\) is uniformly quasi-convex.

Now let \(x\) be a point such that \(d_{CU}(x,\Y) \leq r\) for all \(U \in \S\). Let \(y\) be a point witnessing the closest point projection of \(x\) on \(H_i = \Y\). It easily follows from the above discussion that for each \(U \in \S\), \(|d_{CU}(x,y)- d_{CU}(x, \Y)|\) is uniformly bounded. Then the result follows from the distance formula applied to \(x\) and \(y\).

\textit{The case where \(e\) is an edge of \(\GG -\GG'\) and \(v\) is a vertex of \(\GG\).}
This is a straightforward consequence of Proposition \ref{prop: Y is hierarchically quasi-convex}.

\end{proof}

\end{document}